\documentclass{amsart}
\usepackage{amsmath,amsthm}
\usepackage{amsfonts,amssymb}
\usepackage{accents}
\usepackage{enumerate}
\usepackage{accents,color}
\usepackage{graphicx}
\newcommand{\lvt}{\left|\kern-1.35pt\left|\kern-1.3pt\left|}
\newcommand{\rvt}{\right|\kern-1.3pt\right|\kern-1.35pt\right|}

\newtheorem{thm}{Theorem}[section]

\newtheorem{prop}[thm]{Proposition}

\newtheorem{defn}[thm]{Definition}

\theoremstyle{remark}

 \def\la{{\langle}}
 \def\ra{{\rangle}}
 \def\ve{{\varepsilon}}

 \def\d{\mathrm{d}}
 
 \def\sph{{\mathbb{S}^{d-1}}}

 \def\sP{{\mathsf P}}

 \def\sd{{\mathsf d}}

 \def\t{{\theta}}
 \def\l{{\lambda}}
 \def\o{{\omega}}
 \def\s{\sigma}
 \def\la{{\langle}}
 \def\ra{{\rangle}}
 \def\ve{{\varepsilon}}

 \def\cb{{\mathbf c}}

 \def\CF{{\mathcal F}}
 \def\CH{{\mathcal H}}

 \def\CV{{\mathcal V}}

 \def\BB{{\mathbb B}}

 \def\NN{{\mathbb N}}

 \def\RR{{\mathbb R}}
 \def\SS{{\mathbb S}}
 \def\TT{{\mathbb T}}
 \def\VV{{\mathbb V}}
 \def\XX{{\mathbb X}}
 
      \def\proj{\operatorname{proj}}

\def\lla{\langle{\kern-2.5pt}\langle}      
\def\rra{\rangle{\kern-2.5pt}\rangle}

\newcommand{\wh}{\widehat}

\def\f{\frac}

\graphicspath{{./}}
\begin{document}

\title{Positive definite functions on a regular domain}
\author{Martin Buhmann}
\address{Justus-Liebig University, Lehrstuhl Numerische Mathematik, 35392 Giessen, Germany}
\email{Martin.Buhmann@math.uni-giessen.de}
\author{Yuan~Xu}
\address{Department of Mathematics, University of Oregon, Eugene, 
OR 97403--1222, USA}
\email{yuan@uoregon.edu} 
\thanks{The second author thanks the Alexander von Humboldt Foundation for 
supporting his visit to U. Giessen, during which the work was carried out; he was partially supported by 
Simons Foundation Grant \#849676.}
\date{\today}  

\begin{abstract} 
We define positive and strictly positive definite functions on a domain and study these functions on a list 
of regular domains. The list includes the unit ball, conic surface, hyperbolic surface, solid hyperboloid, 
and simplex. Each of these domains is embedded in a quadrant or a union of quadrants of the unit sphere
by a distance-preserving map, from which characterizations of positive definite and strictly positive definite 
functions are derived for these regular domains.
\end{abstract}

\maketitle

\section{Introduction} 
\setcounter{equation}{0}

A central topic in the approximation of functions in several variables
is the approximation by interpolation. Given real-valued data, call them
$b_i$, $i=1,2,\ldots,N$, and a set of distinct points $\{x_1, x_2,
\ldots, x_N\}\subset\Omega\subset\RR^d$, the task of finding an interpolant from a
given linear space, e.g., spanned by shifts or otherwise of a kernel-function, is
to find $g:\Omega\mapsto\RR$ with 
$$ 
g(x_i)=b_i,\qquad\forall\;1\leq i\leq N.
$$
This $g$ can for instance be formed as a sum of univariate basis functions $\varphi$,
say, combined with a distance function
$\sd_\Omega:\Omega\times\Omega\mapsto\RR$: Then it reads
$$
g(x)=\sum_{i=1}^N\lambda_i \varphi(\sd_\Omega(x,x_i)),\qquad x\in \Omega.
$$
Therefore, the interpolation problem is uniquely solvable as soon as the 
interpolation matrix $[\varphi(\sd_\Omega(x_i,x_j))]_{i,j=1}^N$ is
non-singular. Clearly, if $\varphi$ is such that this matrix is positive
definite, the said condition is fulfilled, but other choices are
possible (see, e.g., \cite{MJQI}). If no other properties of $\Omega$
are required other than that it be a subset of $\RR^d$, the Euclidean
distance function will be an obvious choice for the distance function,
and then the whole construction comes down to interpolation with the
so-called radial basis functions (\cite{MB}). 

Interpolation on spheres and other regular domains is a special case of 
the great theoretical and practical interest of the above general {\it Ansatz\/} 
and the choice of the distance function becomes particularly critical. 
For the unit sphere $\SS^d$, the distance function is the geodesic distance 
$$
  \sd_{\SS}(\xi, \eta) = \arccos \la \xi,\eta\ra, \qquad \xi, \eta \in \SS^d, 
$$
and PDFs on the unit sphere have been studied extensively. The aim of the present 
paper is to study positive definite functions on a bounded regular domain, more general than 
the unit sphere. 

Let $\Omega$ be a bounded domain in $\RR^d$ and let $\sd_\Omega(\cdot,\cdot)$ be a 
distance function on $\Omega$. We assume that the domain is dilated so that $\sd_\Omega(x,y) \le \pi$ for 
all $x,y \in \Omega$. Let $f: [-1,1] \mapsto \RR$ be a continuous function. For $N \in \NN$, let 
$$
      \Xi_N = \{x_1,\ldots, x_N\} \subset \Omega
$$
be a set of distinct points in $\Omega$. We denote by $f[\Xi_N]$ the $N\times N$ matrix
$$
   f[\Xi_N] = \left[ f(\cos \sd_\Omega (x_i,x_j)) \right]_{i,j =1}^N.  
$$
In other words, $\varphi$ from the above becomes $f(\cos(\cdot))$.

\begin{defn}\label{defn:PDF}
A continuous function $f: [-1,1] \mapsto \RR$ is called a positive definite function (PDF) on $\Omega$ if 
for every $N\in \NN$ and set of $N$ distinct points $\Xi_N =
\{x_1,\ldots, x_N\} \subset \Omega$, the  
$N\times N$ matrix $f[\Xi_N]$ is nonnegative definite, and it is called a strictly positive definite function (SPDF)
if $f[\Xi_N]$ is positive definite. 
\end{defn}

When $\Omega$ is the unit sphere, PDFs are first studied in the seminal work \cite{S}, where the 
space of PDFs is characterized as the function $f$ whose Fourier-Gegenbauer coefficients $\hat f_n$ 
are all non-negative. Motivated by the problem of scatted data interpolation, SPDFs on the unit sphere were  
first considered in \cite{XC}, which gives as a sufficient condition that all $\hat f_n >0$, and they are later 
characterized in \cite{CMS} as the function $f$ that has infinitely many even and infinitely many odd positive 
$\hat f_n$. These functions are important in several topics in approximation theory, computational analysis, 
probability, and statistics. They have generated considerable interest and led to numerous studies and 
generalizations. We refer to \cite{BM, BPP, BJ1, BJ2, CMP, GM, GMPa, GMPe, MO} for several recent 
studies that generalize the concept of PDFs and SPDFs in several different directions. There are many 
more papers on PDF functions and their applications. We refer to \cite{Gne} for a survey that includes 
discussions on applications in probability and other fields and to \cite{HGM} for some applications in 
computational analysis. 

In the present paper, we consider regular domains $\Omega$ that include the unit ball, conic surface, hyperbolic surface, 
solid hyperboloid, and simplex, on which PDFs are defined via their
respective distance functions.

It turns out that each of these domains can be mapped to a portion of the unit sphere via a distance-preserving
map, which leads us to consider PDFs on quadrants of the unit spheres, such as the positive quadrant or upper 
hemisphere. This provides a large class of PDFs on the regular domain. Our main result provides a 
characterization of SPDFs on these quadrants and each of the regular domains in our list among this class
of PDFs. 

One of the essential tools for studying PDFs on the unit sphere is the addition formula of spherical harmonics, 
which gives a closed-form formula for the reproducing kernel of the space of spherical harmonics of a fixed 
degree. For a regular domain, the analog of the addition formula is a closed-form formula for the reproducing 
kernel of the space of orthogonal polynomials of fixed degree, where the orthogonality is defined in 
$L^2(\Omega, w)$, with $w$ being an analogy of the Chebyshev weight function on the domain. For each
of our regular domains, an analog of the addition formula is explicitly given, which turns out to be an average
of a PDF function over reflections of a base point. The latter resembles the image method for finding the 
Green's function for the Poisson equation over a quadrant in $\RR^d$.

The paper is organized as follows. The next section is preliminary, where we review background materials and
the results on the unit sphere. We discuss characterizations of PDFs and SPDFs on quadrants of the unit sphere 
in the third section and on our list of regular domains in the fourth section. 

\section{Preliminaries}
\setcounter{equation}{0}
We review the results on the unit sphere in the first subsection, the reproducing kernels of orthogonal 
polynomials in the second subsection, and state a couple of simple properties of PDFs in the third 
subsection.

\subsection{PDFs on the unit sphere}
The characterization of PDFs on the unit sphere relies on the Fourier-Gegenbauer expansions, which we 
review first. For $\l > -\f12$, let the weight function be 
$$
w_\l(t) = (1-t^2)^{\l - \f12}, \quad -1 \le t \le 1. 
$$ 
The Gegenbauer polynomials $C_n^\l$ are orthogonal polynomials of degree $n$ that satisfy 
$$
  c_\l \int_{-1}^1 C_n^\l(t) C_m^\l(t) w_\l(t) \,\d t = h_n^\l \delta_{m,n}, \qquad h_n^\l = \frac{\l}{n+\l} C_n^\l(1),
$$
where $c_\l = \Gamma(\l+\f12)\Gamma(\f12)/\Gamma(\l)$ and $C_n^\l(1) =
(2\l)_n/n!$. The expression $(\cdot)_n$ denotes the Pochhammer symbol. The 
Fourier-Gegenbauer series of $f \in L^2([-1,1],w_\l)$ is defined by 
\begin{equation} \label{eq:FourierGegen}
  f(\cos \t) = \sum_{n=0}^\infty \hat f_n^\l C_n^\l (\cos\t),  \quad \l > -\tfrac12,
\end{equation}
where the coefficients $\hat f_n^\l$ are given by 
$$
     \hat f_n^\l = [h_n^\l]^{-\f12} c_\l \int_0^\pi f(\cos\t) C_n^\l (\cos \t) (\sin\t)^{2 \l} \d \t. 
$$

We can now state the characterization of PDFs on the unit sphere in \cite{S}.

\begin{thm}\label{thm:PDF-Sphere}
Let $d\ge 2$ and $\l = \f{d-1}{2}$. A continuous function $f: [-1,1]\mapsto \RR$ is positive definite on $\SS^d$ 
if and only if $\hat f_n^\l \ge 0$ for all $n$, in which case the series \eqref{eq:FourierGegen} converges 
absolutely and uniformly to $f(\cos\t)$ on $[0,\pi]$. The most general $f$ that is positive definite on $\SS^d$ 
is therefore given by the expansion  
\begin{equation}\label{eq:series-f}
 f(\cos \t) = \sum_{n=0}^\infty a_n C_n^{\f{d-1}2} (\cos\t), \qquad a_n \ge 0, \quad \forall n \in \NN_0.
\end{equation}
\end{thm}
  
For the strictly positive definite functions, we only need to consider $f$ that has nonnegative Fourier-Gegenbauer 
coefficients since $\mathrm{S}\Phi(\SS^d) \subset \Phi(\SS^d)$. For such $f$, we can write \eqref{eq:FourierGegen} as 
\begin{equation} \label{eq:series_CF}
  f = \sum_{n=0}^\infty \hat f_n^\l C_n^\l = \sum_{n \in \CF} \hat f_n^\l C_n^\l, 
\end{equation}
where $\CF$ is the index set defined by
\begin{equation} \label{eq:CF}
  \CF :=\{n \in \NN_0: \hat f_n >  0\}.  
\end{equation}
The characterization of PDFs on the unit sphere in \cite{CMS} is stated as follows.

\begin{thm}\label{thm:SPDF-Sphere}
Let $d\ge 2$ and $\l = \f{d-1}{2}$. A continuous function $f: [-1,1]\mapsto \RR$ is strictly positive definite
if and only if $\CF$ contains infinitely even integers and infinitely odd integers. 
\end{thm}

The proofs of these theorems rely on the properties of spherical harmonics and the Gegenbauer polynomials,
which will also be needed in the next section when we consider PDFs on quadrants of the unit sphere. We 
recall the necessary definitions below. 

When $\l = \frac{d-1}{2}$, the Gegenbauer polynomials are often called ultraspherical polynomials since 
they are associated with spherical harmonics. Let $\CH_n^{d+1}$ be the space of spherical harmonics of 
degree $n$, which are homogeneous polynomials in $d+1$ variables restricted to the sphere $\SS^d$.  
Let $\{Y_\ell^n: 1 \le \ell \le \dim \CH_n^{d+1}\}$ be an orthonormal basis of $\CH_n^d$. Then 
$$
 \frac{1}{\o_d} \int_{\SS^d} Y_\ell^n(\xi) Y_{\ell'} ^m(\xi) \,\d \s
 (\xi) = \delta_{\ell, \ell'} \delta_{n,m},  
$$
where $\d \s$ is the surface measure of $\SS^d$ and $\o_d$ denotes the surface area of $\SS^d$. 
An orthonormal basis of $\CH_n^{d+1}$ can be found explicitly by using spherical coordinates; see, 
for example, \cite[p. 116]{DX}. The reproducing kernel of the space $\CH_n^{d+1}$ is a function 
$\sP_n(\cdot,\cdot)$ such that 
$$
  \frac{1}{\o_d} \int_{\SS^d} \sP_n(\xi,\eta) Q(\eta) \,\d \s(\eta) = Q(\xi), \qquad \forall Q\in \CH_n^d.
$$
Let $\{Y_\ell^n: 1 \le \ell \le \dim \CH_n^d\}$ be an orthonormal basis of $\CH_n^d$. One of the most important 
properties of the spherical harmonics is their {\it addition formula}, which gives
a closed-form formula for the reproducing kernel (cf. \cite[(1.2.8)]{DaiX})  
\begin{equation}\label{eq:additionS}
  \sP_n(\xi,\eta) =  \sum_{\ell =1}^{\dim \CH_n^{d+1}} Y_\ell^n(\xi)Y_\ell^n(\eta) 
  = Z_n^{\f{d-1}{2}}(\la \xi,\eta\ra), \quad \forall \xi,\eta\in \SS^d,
\end{equation}
where $Z_n^\l$ is a constant multiple of the Gegenbauer polynomial, 
\begin{equation}\label{eq:Zn}
   Z_n^\l(t) := \frac{n+\l}{\l} C_n^\l (t), \qquad \l \ge - \tfrac 12.
\end{equation}
We use the notation $\la\cdot,\cdot\ra$ for inner products here.

As an immediate consequence of \eqref{eq:additionS}, for example, it leads to 
$$
  \cb^\mathsf{T} C_n^\f{d-1}{2}[\Xi_N] \cb =  \frac{\l}{n+\l} \sum_{\ell =1}^{\dim \CH_n^{d+1}} 
        \left[\sum_{k=1}^N c_\ell Y_{\ell}^n(\xi_k)\right]^2 \ge 0
$$
for any $\cb \in \RR^N$, which shows that $C_n^{\f{d-1}{2}} \in \Phi(\SS^d)$ for each $n \ge 0$. 
Here, we use the notation $\xi_k$ instead of $x_k$ for our elements of the set $\Xi_N$ as long as 
we are on the sphere.

Writing in terms of the orthonormal basis derived via the separation of variables in spherical coordinates,
\eqref{eq:additionS} leads to the addition formula for the Gegenbauer polynomials. For $\l > 0$, the latter states
\cite[(9.85)]{AAR} 
\begin{equation}\label{eq:addtionC}
 C_n^\l (\cos \t \cos \phi + u \sin \t \sin \phi) = \sum_{k=0}^n b_{k,n}^\l Q_{k,n}^\l(\t)  Q_{k,n}^\l(\phi) C_k^{\l-\f12} (u),
\end{equation}
where the coeffients $b_{k,n}^\l$ are positive constants and 
$$
  Q_{k,n}^\l (\t) = (\sin \t)^k C_{n-k}^{\l+k}(\cos \t), \quad 0 \le k \le n. 
$$

\subsection{Reproducing kernels for orthogonal polynomials}
Let $w$ be a non-negative weight function defined on the compact domain $\Omega \subset \RR^d$. 
Then the bilinear form with weight 
$$
  \la f,g\ra_w = \int_\Omega f(x) g(x) w(x) \,\d x 
$$
defines an inner product of $L^2(\Omega, w)$. We consider orthogonal polynomials for this inner product, 
which are well defined among polynomials restricted to $\Omega$. Notice that when $\Omega$ is a quadratic
surface defined by $\phi(x)=0$, such as the unit sphere, the polynomials restricted on $\Omega$ belong to 
the polynomial ring $\RR[x]/ \la \phi\ra$. 

Let $\CV_n(\Omega)$ be the space of orthogonal polynomials of degree at most $n$. 

\begin{defn} \label{def:reprod}
A function $\sP_n(w;\cdot,\cdot)$ is called a reproducing kernel of $\CV_n(\Omega)$ if 
\begin{enumerate}
\item $\sP_n(w; \cdot,\cdot)$ is symmetric: $\sP_n(w;  x,y) = \sP_n(w; y,x)$; 
\item $\sP_n(w; \cdot,y) \in \CV_n(\Omega)$ for all $y\in \Omega$;  
\item $\sP_n(w; \cdot,\cdot)$ reproduces $\CV_n(\Omega)$ in the sense that
\begin{equation}\label{eq:reprod}
   \int_\Omega \sP_n(w; x,y) Q(y) w(y) \,\d y = Q(x), \qquad Q \in \CV_n(\Omega). 
\end{equation}
\end{enumerate}
\end{defn}
The kernel is uniquely determined by these properties.
Furthermore, if $\{P_\ell^n \}_{\ell=1}^{a(n,d)}$ is an orthonormal
basis of $\CV_n(\Omega)$, then 
$$
    \sP_n (w;x,y) = \sum_{k=1}^{a(n,d)} P_k^n(x) P_k^n(y), \qquad a(n,d) = \dim \CV_n(\Omega).
$$
We refer to \cite[Section 3.6]{DX} for further properties. The reproducing kernels are important for understanding
the Fourier orthogonal expansion. Indeed, for $f\in L^2(\Omega,w)$, its Fourier orthogonal series is defined by
$$
  f = \sum_{n=0}^\infty \sum_{\ell =1}^{a(n,d)} \wh f_{\ell,n} P_\ell^n, \quad \hbox{where} \quad
      \wh f_{\ell,n} = \la f, P_\ell^n \ra_w.
$$ 
Then $\sP_n(w;\cdot,\cdot)$ is the kernel of the projection operator $\proj_n: L^2(\Omega,2) \mapsto \CV_n(\Omega)$,
$$
  \proj_n f (x)=  \sum_{\ell =1}^{a(n,d)} \wh f_{\ell,n} P_\ell^n (x) =
  \int_\Omega f(y) \sP_n(w; x,y) w(y) \,\d y,\qquad x\in\Omega.
$$

For the analysis of the Fourier orthogonal expansion, a closed-form
formula of the reproducing kernel $\sP_n(w;\cdot,\cdot)$ 
is often essential. The addition formula \eqref{eq:additionS} for the unit sphere provides one of the most 
elegant formulas in this regard. We will encounter similar closed-form formulas on our regular domains and 
will refer to them as addition formulas as well. 

Finally, a word about our notation. Namely, we shall use different fonts for the kernels on different types of domains: 
$\sP_n$ if $\Omega$ is a hypersurface and $\sP_n$ if $\Omega$ is a solid domain.

\subsection{PDFs on a domain}
For convenience, we denote $\Phi(\Omega)$ the space of all positive definite functions on $\Omega$ and by 
$\mathrm{S}\Phi(\Omega)$ the space of all strictly positive definite functions on $\Omega$ throughout the 
rest of the paper. 

Several properties for $\Phi(\Omega)$ follow readily from the theory of matrices. 
\begin{enumerate} 
\item If $f_j \in \Phi(\Omega)$ and $c_j \ge 0$ for $1\le j \le n$,
  then $f= c_1 f_1 + \cdots + c_n f_n \in \Phi(\Omega)$. 
\item If $f\in \Phi(\Omega)$ and $g \in \Phi(\Omega)$, then $f g \in \Phi(\Omega)$,
\end{enumerate}
and both properties hold for $\mathrm{S}\Phi(\Omega)$. While the first one is evident from the definition of PDFs, 
the second one follows from Schur's product theorem, which states that the Hadamard product (or entrywise
product) of two positive definite matrices is also a positive definite matrix.

\section{Positive definite functions on quadrants of the unit sphere}
\setcounter{equation}{0}

In this section, we study PDFs and SPDFs on quadrants of the unit sphere. For a positive integer $k$, 
$1 \le k \le d$, we define $\SS^d_{k, +}$ by 
$$
  \SS_{k,+}^d = \left\{\xi \in \SS^d: \xi_k > 0, \ldots, \xi_{d+1} > 0\right\}. 
$$ 
In particular, $\SS_{d+1,+}^d$ is the upper hemisphere of $\SS^d$ and $\SS_{1,+}^d$ is the positive 
quadrant of $\SS^d$. We consider $\Phi(\SS^d_{k,+})$ with $\SS_{k,+}^d$ equipped with the geodesic 
distance $\sd_{\SS^d_{k,+}}(\cdot,\cdot) =\sd_\SS(\cdot,\cdot)$ inherited from $\SS^d$. For the PDF 
functions, we allow the data set $\Xi_N$ to be the subset of $\overline{\SS}_{k,+}^d$ so that it can
include the points on the boundary of $\SS_{k,+}^d$. It follows that $\Phi(\SS^d) \subset  \Phi(\bar \SS^d_{k,+})$. 

The class of PDFs for $\SS^d_{k,+}$ is again given in terms of the Fourier-Gegenbauer series 
\eqref{eq:FourierGegen} and is similar to that of Theorem \ref{thm:PDF-Sphere}. We shall not include the 
convergence of the series in the statements in the rest of the paper. 

\begin{thm}\label{thm: PDF-Sk}
Let $d\ge 2$ and $\l = \f{d-1}{2}$. Let $k$ be an integer such that $1\le k \le d+1$. A continuous function 
$f: [-1,1]\mapsto \RR$ is positive definite on $\bar \SS_{k,+}^d$ if $\hat f_n^\l \ge 0$ for all $n$. Moreover,
if $k=d+1$, then it is also necessary and, in particular, $\Phi(\bar \SS_{d+1,+}) = \Phi(\SS^d)$. 
\end{thm}
  
\begin{proof}
Since $\Phi(\SS^d) \subset  \Phi(\bar \SS^d_{k,+})$, the sufficient direction follows immediately from the 
classical result for $\Phi(\SS^d)$. In particular, $C_n^{\l} \in  \Phi(\bar \SS^d_{k,+})$ for all $n \ge 0$. We now 
prove the other direction for $k = d+1$. Assume that $f$ is a PDF on $\SS_{d+1,+}^d$, so we need 
to show that $\hat f_n^\l \ge 0$. For $\Phi(\bar \SS^d)$, this is established by using the identify
\begin{align*}
 \int_{\SS^d} f(\la \xi,\eta \ra)\,\d \s(\xi) = \frac{1}{\o_d} \int_{\SS^d}  \int_{\SS^d} f(\la \xi, \eta \ra)\,\d \s(\eta)\,\d \s(\xi),
\end{align*}
where $\o_d$ denotes the surface area of $\SS^d$. For $1 \le k \le d+1$ and $\xi \in \SS^d$, we define 
$\xi_{k}^+:= (\xi_1,\ldots,\xi_{k-1}, |\xi_k|, \ldots, |\xi_{d+1}|) \in \bar \SS_{k,+}^d$ and similarly for $\eta_{k}^+$.


The above identity is proved by using the invariant of $\d \s$ on $\SS^d$, which holds if $\eta$ is replaced by 
$\eta_{k}^+$, since there is an rotation $h$ such that $h \eta_k^+ = e$, so that 
\begin{align*}
  \int_{\SS^d}\int_{\SS^d} f(\la \xi,\eta_{k}^+ \ra)\,\d \s(\eta) \d \s (\xi)
  & =   \int_{\SS^d} \int_{\SS^d}  f(\la  h \xi, e \ra)d \s(\eta) \d \s (\xi) \\
  & = \o_d \int_{\SS^d} f(\la \xi, e \ra)\,\d \s(\xi) = \o_d \int_{\SS^d} f(\la \xi, \eta_{k}^+ \ra)\,\d \s(\xi),
\end{align*}
where the second step follows from the rotation invariance of $\d \s$. We claim that if $f$ is even and PDF, then
\begin{equation} \label{eq:intSS=intS}
 \o_d \int_{\SS^d} f(\la \xi, e\ra) \d\s(\xi) = 4 \int_{\bar \SS_{d+1,+}^d}\int_{\bar \SS_{d+1,+}^d} f(\la \xi,\eta \ra)\,\d \s(\eta) \d \s (\xi) \ge 0.
\end{equation}
For simplicity, we write $\SS_+^d = \bar \SS_{d+1,+}^d$ in the proof below. Let 
$\SS_{-}^d = \{\xi \in \SS^d: \xi_{d+1} < 0\}$ and $\xi_{d+1,-} = (\xi', -\xi_{d+1})$ for $\xi \in \SS^d$.

Then
\begin{align*}
 \int_{\SS^d}\int_{\bar \SS_{-}^d} f(\la \xi,\eta_{d+1,+} \ra)\d \s(\xi) \d \s (\eta) & = 
  \int_{\SS^d} \int_{\SS_{+}^d} f(\la \xi_{d+1,-},\eta_{d+1,+} \ra) \d \s (\xi) \d \s(\eta) \\
  & = \int_{\SS^d} \int_{\SS_{+}^d} f(- \la \xi,\eta_{d+1}^+ \ra)\, \d \s (\xi) \d \s(\eta) \\
   & = \int_{\SS^d} \int_{\SS_{+}^d} f(\la \xi ,\eta_{d+1}^+ \ra)\, \d \s (\xi) \d \s(\eta),
\end{align*}
where the second identity follows from changing variables $\eta_j \mapsto -\eta_j$ for $1 \le j \le d$ and
the last identity uses $f$ being even. Putting these integrals together, we have proved
\begin{align*}
  \int_{\SS^d}\int_{\SS^d} f(\la \xi,\eta_{k}^+ \ra)\,\d \s(\eta) \d \s (\xi) = 2 
    \int_{\SS^d} \int_{\SS_{+}^d} f(\la \xi ,\eta_{d+1}^+ \ra)\, \d \s (\xi) \d \s(\eta),
\end{align*}
from which the identity of \eqref{eq:intSS=intS} follows readily, since $\eta_{d+1}^+$ is even in its $(d+1)$st variable. 
Since $f \in \Phi(\bar \SS_{d+1,+}^d)$, applying a positive cubature rule of degree $n$ on $\bar \SS_{d+1,+}^d$ shows that 
the right-hand side is nonnegative for all polynomials of degree $\le n$ and, by the Weierstra\ss{}  theorem 
about the density of polynomials, for all continuous functions. Hence, we have proved that the integral in 
\eqref{eq:intSS=intS} is non-negative. 

Now, for $d \ge 2$, a change of variable shows that 
$$
    \int_{\SS^d} f(\la \xi, \eta \ra) \,\d \s(\eta) = \o_{d-1} \int_{0}^\pi f(\cos \t)(\sin \t)^{d-1} \,\d \t.
$$
Hence, setting $e = e_{d+1}=(0,\ldots,0,1)\in\RR^{d+1}$ as a $(d+1)$-dimensional vector and using the
invariance of $\d \s$, we see that 
$$
      \hat f_n^\l = [h_n^\l]^{-\f12} \frac{1}{\o_d} \int_{\SS^d} f(\la e, \xi \ra) C_n^\l(\la  e,\xi \ra) \,\d \s(\xi).
$$
Since both $f$ and $C_n^\l$ are in $\Phi(\bar \SS_{d+1,+}^d)$, so is their product. Hence, by \eqref{eq:intSS=intS},
we conclude that $\hat f_n^\l \ge 0$. This completes the proof. 
\end{proof}
  
We do not know if there is any element in $\Phi(\bar \SS_{k,+}^d)$ that is not of the form \eqref{eq:series-f}. 
  
Since an SPDF is automatically PDF, we now consider SPDFs given by the Fourier-Gegenbauer series 
\eqref{eq:FourierGegen} with $\hat f_n^\l \ge 0$, which can be written as in \eqref{eq:series_CF} with 
the sum over the index set $\CF =\{n \in \NN_0: \hat f_n >  0\}$ defined in \eqref{eq:CF}.

\begin{thm}\label{thm: SPDF-Sk}
Let $d\ge 2$ and $\l = \f{d-1}{2}$. Let $k$ be an integer such that $1\le k \le d+1$. A continuous function 
$f: [-1,1]\mapsto \RR$ in $\Phi(\SS^d)$ is strictly positive definite on $\SS_{k,+}^d$ if and only if the 
set $\CF$ satisfies 
\begin{itemize}
\item[(i)] $k=1:$ $\CF$ is infinite. 
\item[(i)] $k=2:$ $\CF$ is infinite and contains at least one even and one odd element. 
\item[(iii)] $2 \le k \le d+1:$ $\CF$ contains infinitely many even and infinitely many odd elements. 
\end{itemize}
\end{thm}

\begin{proof}
For $f[\Xi]$ to be positive definite, we need to consider {\it antipodal pairs} of points, which are two points $\xi_i$ 
and $\xi_j$ such that $\xi_i = - \xi_j$. If $\Xi_N$ contains the antipodal pairs, then the two rows of $f[\Xi_N]$ 
corresponding to the two points are the same if $f$ is an even function and they differ by a sign if $f$ is an odd 
function. In either case, the matrix  $f[\Xi_N]$ is not positive definite. For $k =1$, there is no antipodal pair in 
$\SS_1^{d} = \SS_+^{d}$. For $k =2$, there is exactly one pair, $e_1 = (1,0, \ldots, 0)$ and $- e_1$, in 
$\SS_2^{d}$. If $\CF$ contains no odd element, then $f$ is an even function, whereas if $\CF$ contains no 
even element, then $f$ is odd. Hence, $\CF$ must contain one even and one odd element for $k = 2$. 
For $k > 2$, there are infinitely many antipodal pairs in $\SS_k^d$.

We choose $\Xi_{2N} = \{x_1,\ldots, x_{2N}\}
\subset \SS_{k,+}^d$ so that if $x_i \in \Xi_{2N}$, then $-x_i \in \Xi_{2N}$. Writing $f  = f_e + f_o$, where $f_e$ is
the even part of $f_o$ is the odd part of $f$, we see that the same consideration implies that
$$
   \mathrm{rank}\, f_e[\Xi_{2N}] \le N \qquad \hbox{and} \qquad    \mathrm{rank}\, f_o[\Xi_{2N}] \le N. 
$$
Assume now that $\CF$ contains only finite many even elements. Let $M$ be the largest integer for which 
$\hat f_{2n}=\hat f_{2M} \ne 0$. Since $C_n^\l$ is a polynomial of degree $n$ and it has the same parity as $n$, we can write 
$$
  f(t) = \sum_{n =0}^M\hat f_{2n} C_{2n}^\l(t) = \sum_{n=0}^{M} b_n t^{2n} 
$$
for some set $\{b_n\}$ of coefficients with $b_0 \ne 0$. For $\ell \in \NN$, let $G_\ell$ be the matrix 
$G_\ell = [\la \xi_i, \xi_j \ra^\ell]_{i,j =1}^N$. For $\ell =1$, $G_\ell$ is a 
Gram matrix, which is nonnegative definite and $\mathrm{rank}\, G_1 \le d+1$, since $\xi_j \in \SS^d$.  
By Schur's product theorem, the matrices $G_\ell$ are also nonnegative definite and their ranks are bounded by
$\mathrm{rank}\, G_\ell \le (\mathrm{rank}\, G_1)^\ell \le (d+1)^\ell$. Consequently, since 
$\mathrm{rank}\, (A+ B) \le \mathrm{rank}\,A + \mathrm{rank}\,B$ for general matrices $A$ and $B$, we conclude that 
$$
  \mathrm{rank}\,f[\Xi_N] \le  \sum_{m=0}^{M/2} \mathrm{rank}\, G_{M-2m} \le \sum_{m=0}^{M/2} (d+1)^{M-2m}, 
$$  
which is independent of $N$. Thus, for $N > \sum_{m=0}^{M/2} (d+1)^{M-2m}$, the matrix $f[\Xi_N]$ does 
not have full rank and, hence, $f$ is not strictly positive definite. A similar proof applies to the case when $\CF$
contains only finitely many odd elements.

In the other direction, we assume that $f$ is a strictly positive definite function on $\SS_{k,+}^d$ with an infinite 
set $\CF$. For $k$ satisfying $1\le k \le d+1$, let $\Xi_N = \{\xi_1,\ldots, \xi_N\}$ be a set of $N$ distinct points 
in $\SS_{k,+}^d$. Since $\Xi_N$ is finite and contains no points that are on the boundary as $\SS_{k,+}^d$ is open, 
we can choose a $p \in \SS^d$ as the ``north pole" of a new coordinate system so that the set $\Xi_N$ is in 
$\SS_{k,+}^d$ using the new coordinate system. There are infinitely many such $p$. Since $\Xi_N$ is finite, we can
choose a $p$ such that it is not the normal vector of the plane spanned by $x_i$ and $x_j$ for all $1 \le i \ne j \le N$, 
which means that $\la p, \xi_i\ra \ne  \la p, \xi_j\ra$ if $i \ne j$ so that $\{\la p,\xi_j\ra: 1 \le i \le N\}$ forms a set 
of $N$ distinct numbers in $(-1,1)$. 
Without loss of generality, we assume $p = (0,\ldots,0,1)$. Thus, every point $\xi_i$ has a representation 
of the form 
$$
  \xi_i = (\xi_i' \sin \t_i, \cos\t_i), \qquad \xi_i' \in \sph, \quad \cos \t_i = \la p, \xi_i\ra. 
$$
By the addition formula \eqref{eq:addtionC} of the Gegenbauer polynomials, we conclude that 
$$
  C_n^\l(\la \xi_i,\xi_j\ra) = \sum_{k=0}^n b_{k,n}^\l Q_{k,n}^\l (\t_i) Q_{k,n}^\l (\t_j) C_k^{\l-\f12} (\la \xi_i', \xi_j'\ra),
$$
which implies, in particular, that for $\xi = (\xi' \sin\t, \cos \t) \in \SS^d$, 
$$
  \sum_{j=1}^N c_j f(\la x,x_j\ra) = \sum_{n=0}^\infty 
     \hat f_n^\l \left[\sum_{j=1}^N c_j Q_{k,n}^\l(\t_j)C_k^{\l-\f12}(\la \xi_j',\xi' \ra)\right] Q_{k,n}^\lambda (\t).
$$
Since $\SS_{k,+}^d \subset \SS^d$, it is known that the matrix $f[\Xi_N]$ is positive definite if and only if 
the set of functions $\{f(\la \xi,\xi_i\ra): 1 \le i\le N\}$ is linearly independent (cf.\ \cite[Theorem 14.3.5]{DaiX}). 
Consequently, it is sufficient to prove that if 
\begin{equation} \label{eq:indep}
\sum_{j=1}^N c_j Q_{k,n}^\l(\t_j)C_k^{\l-\f12}(\la \xi_j', \xi' \ra) =0, \quad 0 \le k \le n, \quad n \in \CF,
\end{equation}
then $c_1 = c_2 = \cdots =c_n =0$. This is proved by induction on $N$ just like the proof for 
Theorem \ref{thm:SPDF-Sphere}, so much so that, for each fixed $k$, the proof given in \cite[p. 380]{DaiX} 
can be followed almost verbatim for $\SS_{k,+}^d$. 
\end{proof}
 
For the strictly positive definite functions, we assume that $\SS_{k,+}^d$ is open so that $\Xi_N$ does not include points 
on the boundary. As an example, let us consider the data set $\Xi_N$ of $N$ points in the closure of $\SS_{2,+}^2$ that 
contains $M_1$ points of the form $(a_1,a_2,0)$ and $M_2$ points of the form $(a_1,0,a_3)$, which are on the boundary 
of $\SS_2^{2,+}$, and the coordinates vector $e_1, e_2, e_3$ of $\RR^3$. Because the set includes all three coordinate 
vectors, we cannot introduce another coordinate system and have the data set on $\SS_{2,+}^2$ in the new coordinate 
system. If we choose $p = e_3$, then all $M_1$ points of the form $(a_1,a_2,0)$ on the boundary will have $\la p, x\ra =0$, 
so that $\{\la p, x_j\ra: 1 \le j \le N\}$ is no longer a set of $N$ distinct points in $(-1,1)$. Consequently, the above proof cannot 
proceed as it is. We do not know if this restriction is essential or if it can be removed with a different proof. 

\section{Positive definite functions on several regular domains}
\setcounter{equation}{0} 

We consider PDFs when $\Omega$ is one of several regular domains, discussed in the consecutive subsections. 
They contain balls, surfaces, and solids of the hyperbolic set, which should be contrasted with the earlier work 
we cited in the introduction, much of which was on spheres. 

\subsection{PDF on the unit ball}
On the unit ball $\BB^d = \{x \in \RR^d: \|x\| \le 1\}$, the distance function is given by  
$$
 \sd_\BB(x,y) = \arccos \left(\la x, y \ra + \sqrt{1-\|x\|^2}  \sqrt{1-\|y\|^2} \right), 
$$
which is a projection of the geodesic distance of the hemisphere $\SS_{d+1,+}^d$ on $\BB^d$.
Thus, we see that a function $f: [-1,1] \mapsto \RR$ is a PDF on $\BB^d$ if 
$$
f[\Xi_N] = \left [f \left(\la x_i,x_j\ra+
  \sqrt{1-\|x_i\|^2} \sqrt{1-\|x_j\|^2} \right)\right]_{i,j=1}^N
$$
is nonnegative definite for all subsets $\Xi = \{x_1,\ldots, x_N\}
\subset \BB^d$.   (For the moment, we do not require these points in
the ball to be distinct yet, this will begin to play a r\^ole when we
shall speak about strict positive definiteness.)

It is well-known that there is a one-to-one correspondence between $\BB^d$ and the upper hemisphere
$\SS_{d+1,+}^d$ defined by 
\begin{equation} \label{eq:mapB-S+}
    \BB^d\ni x \mapsto  \left(x, \sqrt{1-\|x\|^2}\right) \in \SS_{d+1,+}^d,
\end{equation}
and the map preserves the distance between the points from the two domains; more precisely,
$$
   \sd_{\BB}(x,y) = \sd_{\SS} \left (\left(x, \sqrt{1-\|x\|^2}\right), \left(y, \sqrt{1-\|y\|^2}\right) \right). 
$$
Hence, the PDF functions on the unit ball are characterized by applying Theorem \ref{thm: PDF-Sk} with 
$\bar \SS_{d+1,+}^d$. Instead of stating the full analog of Theorem \ref{thm: PDF-Sk}, we state a simple 
version for the record. 

\begin{thm}\label{thm:PDF-Ball}
Let $d\ge 2$. All functions $f: [-1,1] \mapsto \RR$ in $\Phi(\BB^d)$ are given by the expansion  
$$
 f(\cos \t) = \sum_{n=0}^\infty a_n f_n^{\l} C_n^{\l} (\cos\t), \qquad a_n \ge 0, \quad \forall n \in \NN_0.
$$
Moreover, such an $f$ is in $\mathrm{S}\Phi(\mathring{\BB}^d)$ if and only if there are infinitely 
many positive $a_{2n}$ and infinitely many positive $a_{2n+1}$.  
\end{thm}

The ball is the first, and the simplest, of our regular domains $\Omega$. It is worthwhile to discuss some relevant 
aspects of the connection to the unit sphere in this simple case which
has been studied before. 

Under the mapping \eqref{eq:mapB-S+}, the surface measure $\d\s(x)$ on the unit sphere corresponds
to $W_0(x)\, \d x$ on the unit ball, where the positive function $W_0$ is the Chebyshev weight function 
$$
 W_0(x) = \frac{1}{\sqrt{1-\|x\|^2}}, \qquad x \in \BB^d, 
$$
on the unit ball. Let $\CV_n(\BB^d)$ be the space of orthogonal polynomials of degree $n$ with respect to 
$W_0$ on $\BB^d$. Let $\sP_n(W_0; \cdot,\cdot)$ be the reproducing kernel of $\CV_n(\BB^d)$, as defined
in Definition \ref{def:reprod}. 
This kernel satisfies a closed-form formula, again called an addition formula, given by (cf.\ \cite[Corollary 4.2.9]{DX})
\begin{align}\label{eq:additionB}
  \sP_n(W_0; x,y) = \frac12 \left[\sP_n^+ (W_0; x,y) + \sP_n^-(W_0; x,y) \right],
\end{align} 
where $\sP_n^\pm(W_0; \cdot,\cdot)$ is defined by 
$$
     \sP_n^\pm(W_0;x,y) = Z_n^{\f{d-1}{2}}\left(\la x,y\ra\pm \sqrt{1-\|x\|^2}\sqrt{1-\|y\|^2}\right). 
$$   
In particular, $\sP_n^+(W_0; x,y) = Z_n^{\f{d-1}{2}}(\cos \sd_\BB(x,y))$. It is not difficult to verify that the
reproducing property \eqref{eq:reprod} holds if $\sP_n(W_0;\cdot,\cdot)$ is replaced by either $\sP_n^+(W_0;\cdot,\cdot)$ or 
$\sP_n^-(W_0;\cdot,\cdot)$.

However, in order to be a reproducing kernel for $\CV_n(\BB^d)$, the kernel needs also
satisfy the item (2) of Definition \ref{def:reprod} which requires
that it be an element of $\CV_n(\BB^d)$  
as a function of $x$ for each $y$. While $\sP_n^{\pm}(W_0; \cdot,\cdot)$ does not satisfy the latter 
requirement, their sum $\sP_n(W_0; \cdot,\cdot)$ in
\eqref{eq:additionB} does.

Moreover, setting 
$$
\xi = \left(x,\sqrt{1-\|x\|^2}\right) \quad \hbox{and} \quad \eta^\pm  = \left(y, \pm \sqrt{1-\|y\|^2}\right),
$$
then the sum is the average of $Z_n^{\f{d-1}{2}}(\cos \sd_\SS(\xi, \eta^+))$ and 
$Z_n^{\f{d-1}{2}}(\cos \sd_\SS(\xi,\eta^-))$, where $\eta^-$ is the
reflection of $\eta^+$ over the plane defined by the condition  
$\xi_{d+1} =0$. This connection of PDF to the reproducing kernel will appear for the other domains 
that we shall discuss below as well. 

The ball $\BB^d$ is the projection of the upper hemisphere $\SS_{d+1,+}^d$ under the mapping \eqref{eq:mapB-S+}. 
We can also work with the lower hemisphere $\SS_{d+1,-} = \{\xi \in \SS^d: \xi_{d+1} < 0\}$ using the map
$$
  \SS^d \ni x\mapsto \left(x, - \sqrt{1-\|x\|_2}\right) \in \SS_{d+1,-}^d,
$$
which again preserves the distance between the points from the two sets.

It may be tempting to map a set of points on $\BB^d$ to $\SS^d$ by using both mappings at the same time. 
For example, if $\Xi_{2N}$ is a set of $2N$ distinct points in $\BB^d$, such that $\Xi_{2N} = \Xi_{N}^+ \cup \Xi_{N}^-$ 
and $\Xi_{N}^+ = - \Xi_{N}^-$, then we can map the points in $\Xi_N^+$ to the upper hemisphere and points in 
$\Xi_N^-$ to the lower hemisphere.

However, although both maps to hemispheres preserve the distance, the total 
collection of points does not. Indeed, if $x_i \mapsto X_i = (x_i,\sqrt{1-\|x_i\|^2}) \in \SS_{d+1,+}^d$ and 
$x_j\mapsto X_j = (x_j, - \sqrt{1-\|x_j\|^2}) \in \SS_{d+1,-}^d$ are
different points, then 
$$
\cos \sd_\SS(X_i,X_j) = \la x_i, x_j\ra - \sqrt{1-\|x_i\|^2}\sqrt{1-\|x_j\|^2} \ne \cos \sd_\BB(x_i,x_j). 
$$

\subsection{PDF on the hyperbolic surface}
Up to an affine change of variable, the hyperbolic surface that we
shall consider in this subsection can be normalized to
satisfy, for a constant $\varrho \ge 0$, 
$$
  {}_{\varrho}\XX_0^{d+1} :=\left \{(x,t): \|x\| = \sqrt{t^2 - \varrho^2}, \quad  
       x \in \RR^d,\, \varrho \le |t| \le \sqrt{1+\varrho^2}\right \},
$$
which degenerates to a double cone if $\varrho = 0$.  The surface is naturally divided into upper and lower
hyperbolic surfaces, which we denote by 
$$
   {}_{\varrho}\XX_0^{d+1} =  {}_{\varrho}\XX_{0,+}^{d+1} \cup  {}_{\varrho}\XX_{0,-}^{d+1}, \quad\hbox{where}\quad
    {}_{\varrho}\XX_{0,\pm}^{d+1} =  \left \{ (x,t) \in {}_{\varrho}\XX_0^{d+1}: \mathrm{sign} (t) = \pm 1\right\}.
$$
For $\varrho > 0$, the upper and lower parts are disjoint, so we only consider the distance for points in 
either the upper or the lower hyperbolic surfaces. For $(x,t), (y,s)$ both in $ {}_{\varrho}\XX_{0,+}^{d+1}$ or
both in ${}_{\varrho}\XX_{0,-}^{d+1}$,  their distance is defined by \cite{X22}
$$
  \sd_{\XX_0}^\varrho \big( (x,t), (y,s) \big) = \arccos \left(\la x,y \ra + \sqrt{1+\rho^2 - t^2}  \sqrt{1+ \varrho^2 - s^2} \right).  
$$
In particular, for $\Xi_N = \{(x_j,t_j): 1 \le i \le N\} \subset {}_{\varrho}\XX_0^{d+1}$, the matrix $f[\Xi_N]$ is
given by 
$$
     f[\Xi_N] = \left [f \left(\la x_i,x_j\ra+ \sqrt{1+\varrho^2- t_i^2} \sqrt{1+ \varrho^2 - t_j^2} \right)\right]_{i,j=1}^N. 
$$

Like the case of the unit ball, there is a one-to-one correspondence between ${}_\varrho\XX_{0,+}^{d+1}$ and 
the upper hemisphere $\SS_{d+1,+}^d$, defined by
\begin{equation} \label{eq:mapX0-S+}
     {}_{\varrho}\XX_{0,+}^{d+1} \ni (x,t) \mapsto \left(x, \sqrt{1+\varrho^2 - t^2} \right) \in \SS_{d+1,+}^d,
\end{equation}
that preserves the distance on the domains; that is, 
$$
  \sd_{\XX_0}^\varrho\big((x,t), (y,s)\big)= \sd_{\SS} \left( \left(x, \sqrt{1+\varrho^2-t^2} \right),  
      \left(y, \sqrt{1+ \varrho^2 -s^2} \right)\right). 
$$
Moreover, the same mapping also defines a correspondence between $\XX_{0,-}^{d+1}$ and $\SS_{d+1,+}^d$. 

As a result, we obtain a characterization of the positive definite 
functions on the upper, or lower, hyperbolic surface from Theorem \ref{thm: PDF-Sk}. As in the case of the unit 
ball, we state a simplified version for the record. 

\begin{thm}\label{thm:PDF-X0}
Let $d\ge 2$. All functions $f: [-1,1] \mapsto \RR$ in $\Phi({}_\varrho\XX_{0, \pm}^{d+1})$ are given by the expansion  
$$
 f(\cos \t) = \sum_{n=0}^\infty a_n C_n^{\f{d-1}2} (\cos\t), \qquad a_n \ge 0, \quad \forall n \in \NN_0.
$$
Moreover, such an $f$ is in $\mathrm{S}\Phi(\BB^d)$ if and only 
if there are infinitely many positive $a_{2n}$ and infinitely many positive $a_{2n+1}$.  
\end{thm}
 
Under the mapping \eqref{eq:mapX0-S+}, the surface measure $\d\s(t)$ on $\SS^d$ is mapped to  
$$
w_0(t)\, \d \s_{\XX_0} (x,t),
$$ 
where $w_0$ is the Chebyshev weight $w_0(t) = 1/\sqrt{1-t^2}$ and $\d\s_{\XX_0}$ denotes the Lebesgue 
measure on ${}_\varrho\XX_0^{d+1}$. The orthogonal polynomials in $L^2({}_\varrho\XX_0^{d+1}, w_0)$ 
are divided into two families, one contains those that are even in the
$t$ variable and the other contains those that are odd in the $t$ variable \cite{X21a}. Let 
$\CV_n^E({}_\varrho\XX_0^{d+1})$ be the space of orthogonal polynomials of degree $n$ that are even
in the $t$ variable. An explicit orthogonal basis of the space is given in \cite{X21a}. Let $\sP_n^E(w_0; \cdot,\cdot)$ 
be the reproducing kernel of $\CV_n^E({}_\varrho\XX_0^{d+1})$ in $L^2({}_\varrho\XX_0^{d+1}, w_0)$. The 
kernel satisfies an addition formula given by \cite[(3.5)]{X21a} 
\begin{align}\label{eq:additionX0}
  \sP_n^E\big(w_0; (x,t), (y,s)\big) = \frac12 \left[\sP_n^+\big(w_0; (x,t), (y,s)\big) + \sP_n^-\big(w_0; (x,t), (y,s)\big) \right],
\end{align} 
where $\sP_n^\pm\big(w_0; \cdot,\cdot\big)$ is defined by 
$$
\sP_n^\pm \big(w_0; (x,t), (y,s)\big) = Z_n^{\f{d-1}{2}}\left(\la x,y\ra \mathrm{\,sign}(t s) \pm \sqrt{1- t^2}\sqrt{1-s^2}\right).
$$   
If $(x,t)$ and $(y,s)$ are both in ${}_\varrho\XX_{0,+}^{d+1}$ or both in ${}_\varrho\XX_{0,-}^{d+1}$, then
$\sP_n^+ \big(w_0; (x,t), (y,s)\big) = Z_n^{\f{d-1}{2}}\big(\sd_{\XX_0}^\varrho((x,t),(y,s))\big)$. Like the case 
in the unit ball, $\sP_n^\pm (w_0;\cdot,\cdot)$ satisfies the reproducing property, but it is not a polynomial, 
whereas the sum in \eqref{eq:additionX0} is. 

\subsection{PDF on the solid hyperboloid}
For $\varrho \ge 0$, the solid hyperboloid is bounded by the
hyperbolic surface $ {}_{\varrho}\XX_0^{d+1}$ and  
hyperplanes defined by the condition $t = \pm \sqrt{1+\varrho^2}$, 
$$
  {}_{\varrho}\XX^{d+1} :=\left \{(x,t): \|x\| \le \sqrt{t^2 - \varrho^2}, \quad  
        x \in \RR^d,\, \varrho \le |t| \le \sqrt{1+\varrho^2}\right \},
$$
which degenerates to the double cone if $\varrho = 0$. The hyperboloid is naturally divided into upper and lower 
part, which we denote by 
$$
   {}_{\varrho}\XX^{d+1} =  {}_{\varrho}\XX_{+}^{d+1} \cup  {}_{\varrho}\XX_{-}^{d+1}, \quad\hbox{where}\quad
    {}_{\varrho}\XX_{\pm}^{d+1} =  \left \{ (x,t) \in {}_\varrho\XX^{d+1}: \mathrm{sign} (t) = \pm 1\right\}.
$$
Like the case of hyperbolic surface, we only consider the distance for points in either the upper or the lower 
hyperboliod. For $(x,t), (y,s)$ both in $ {}_{\varrho}\XX_{+}^{d+1}$ or both in ${}_{\varrho}\XX_{-}^{d+1}$, 
 their distance is defined by \cite[Definition 4.1 and (4.6)]{X22}
\begin{align*}
  \sd_{\XX}^\varrho \big( (x,t), (y,s) \big)
     = \arccos &\left(\la x,y \ra + \sqrt{t^2-\varrho^2-\|x\|^2}\sqrt{s^2-\varrho^2-\|y\|^2} \right. \\
                     & \qquad\qquad\qquad\qquad  \left. + \sqrt{1+\rho^2 - t^2}  \sqrt{1+ \varrho^2 - s^2} \right),
\end{align*}
which is a mixture of the distances on the unit ball and on the hyperbolic surface. There is a one-to-one  
correspondence between ${}_\varrho\XX_{+}^{d+1}$ and $\SS_{d,+}^d =\{\xi \in \SS^d: \xi_d >0, \xi_{d+1} > 0\}$, defined by
\begin{equation} \label{eq:mapX-S+}
    {}_{\varrho}\XX_{+}^{d+1} \ni (x,t)\mapsto \left(x, \sqrt{t^2 - \varrho^2-\|x\|^2}, \sqrt{1+\varrho^2 - t^2} \right)=: \xi(x,t) 
      \in \SS_{d,+}^d.
\end{equation}
This preserves the distance between the points in the domains; that is, 
$$
  \sd_{\XX}^\varrho\big((x,t), (y,s)\big)= \sd_{\SS} \big(\xi(x,t) , \xi(y,s)\big). 
$$
Likewise, we can define an analog correspondence between ${}_{\varrho}\XX_{-}^{d+1}$ and $\SS_{d,+}^d$. 
Consequently, we obtain a characterization of the positive definite functions on the upper, or lower, hyperboloid
from Theorem \ref{thm: PDF-Sk} for $\SS_{d,+}^d$. Again, we record a simplified version of the characterization. 

\begin{thm}\label{thm:PDF-X}
Let $d\ge 1$. All functions $f: [-1,1] \mapsto \RR$ given by the expansion  
$$
 f(\cos \t) = \sum_{n=0}^\infty a_n C_n^{\f{d-1}2} (\cos\t), \qquad a_n \ge 0, \quad \forall n \in \NN_0,
$$
are in $\Phi({}_\varrho\XX_{\pm}^{d+1})$. And such an $f$ is in $\mathrm{S}\Phi({}_\varrho\XX_{\pm}^{d+1})$ 
if and only if  
\begin{itemize}
\item[(i)] $d=2:$ there are infinite many positive $a_n$ with at least one $a_{2n}$ and at least one $a_{2n+1}$. 
\item[(ii)] $d > 2:$ there are infinite many positive $a_{2n}$ and infinitely many positive $a_{2n+1}$. 
\end{itemize}
\end{thm}
 
Under the mapping \eqref{eq:mapX-S+}, the measure on ${}_\varrho\XX_{+}^{d+1}$ that corresponds to 
the surface measure $\d\s$ on $\SS^d$ is $W_{0,0}(x,t)\, \d x \,\d t$, where $W_{0,0}$ is given by 
$$
   W_{0,0}(x,t) = \frac{|t|}{\sqrt{1- t^2} \sqrt{1-\|x\|^2}}. 
$$
As in the case of hyperbolic surface, the orthogonal polynomials in $L^2({}_\varrho\XX^{d+1}, W_{0,0})$ 
are divided into two families, depending on their parity in the $t$ variable. Let $\CV_n^E({}_\varrho\XX^{d+1})$
be the space of orthogonal polynomials of degree $n$ that are even in the $t$ variable. Then the 
reproducing kernel $\sP_n^E(W_{0,0}; \cdot,\cdot)$ of $\CV_n^E({}_\varrho\XX^{d+1})$ satisfies an
addition formula given by \cite[Corollary 5.6 and (5.16)]{X21a} 
\begin{align}\label{eq:additionX}
\sP_n^E\big(W_{0,0}; (x,t), (y,s)\big) = \frac14 & \left[\sP_n^{+,+}\big((x,t), (y,s)\big) + \sP_n^{+,-}\big((x,t), (y,s)\big)\right.\\
         &  \left. + \sP_n^{-,+}\big((x,t), (y,s)\big) + \sP_n^{-,-}\big((x,t), (y,s)\big)\right], \notag
\end{align} 
where $\sP_n^{\pm, \pm}(\cdot,\cdot)$ is defined by 
\begin{align*}
\sP_n^{\pm,\pm} \big((x,t), (y,s)\big)  & = Z_n^{\f{d-1}{2}}\left(\left(\la x,y\ra \pm \sqrt{t^2-\varrho^2-\|x\|^2}
       \sqrt{s^2-\varrho^2-\|y\|^2}\right) \mathrm{sign}(t s) \right.  \\
   & \hskip 1.9in  \left.   \pm \sqrt{1+\varrho^2- t^2}\sqrt{1+ \varrho^2-s^2}\right).
\end{align*}   
As in the previous cases, the kernel $\sP_n^{+,+}(\cdot,\cdot)$ is the
same as $Z_n^{\f{d-1}{2}}(\cos  \sd_{\XX}^\varrho((x,t),(y,s)))$
so that it is a positive definite kernel on the upper, or lower, hyperboloid ${}_\varrho\XX_{\pm}^{d+1}$, and it 
satisfies the reproducing property. The reproducing kernel given in \eqref{eq:additionX} comes from the average 
that makes it a kernel of polynomials. Furthermore, the average is over $Z_n^{\f{d-1}{2}}\big(\cos  \sd_{\SS}^\varrho(\xi(x,t),\xi^{\pm,\pm}(y,s))\big)$, where $\xi^{\pm,\pm} = (\xi_1,\ldots, \xi_{d-1}, \pm \xi_d, \pm \xi_{d+1})$ are reflections
of $\xi$ with respect to the hyperplanes $\xi_d =0$ and $\xi_{x+1} = 0$.  

\subsection{PDF on the conic surface}
Up to an affine transformation, we consider the conic surface defined by 
$$
  \VV_0^{d+1} := \left \{(x,t): \|x\| = t, \;  x \in \BB^d, \; t \in [0,1]\right \}. 
$$
The distance function in this domain is defined by  \cite[(4.1)]{X21b}
$$
 \sd_{\VV_0}\big((x,t),(y,s)\big) = \arccos \left( \sqrt{\frac{ts + \la x,y \ra}{2}} + \sqrt{1- t}  \sqrt{1- s} \right),
$$
which is different from that on the double cone ${}_0\XX_0^{d+1}$. We do not know if there is a 
distance preserving map from $\VV_0^{d+1}$ to a subset of $\SS^d$ for $d > 2$. For $d=2$, however,
we have such a map between $\VV_0^3$ and $\SS_{2,+}^2$. The latter is given by
$$
     \SS_{2,+}^2 = \{(\xi_1,\xi_2,\xi_3)\in\SS^2: \xi_2 > 0, \xi_3 > 0\},
$$
which is the half of the upper hemisphere in $\RR^3$.

\begin{prop}
The mapping $\xi: \VV_0^3 \mapsto \SS_{2,+}^2$ defined by 
\begin{equation}\label{eq:V0-Smap}
\xi: (x_1,x_2,t) \in \VV_0^3 \mapsto \left (\sqrt{\frac{t-x_1}2}\, \mathrm{sign}(x_2), \sqrt{\frac{t+x_1}2}, \sqrt{1-t} \right)
   = :\xi(x,t)  \in \SS^2_{2,+}
\end{equation}
is a one-to-one correspondence if $x_2 \ne 0$ and it preserves the distances between the points of the domains; that is,
\begin{equation*} 
   \sd_{\VV_0} ((x,t),(y,s)) = \arccos \left(\la \xi(x,t), \xi(y,s)\ra\right ) =  \sd_{\SS} \left(\xi(x,t),\xi(y,s)\right). 
\end{equation*}
\end{prop}

\begin{proof}
It is easy to verify that $\xi_{x,t}\in \SS_{2,+}^2$. The inverse of the map is given by
$$
 \SS_{2,+}^2\ni\xi  \mapsto \left(\xi_2^2 - \xi_1^2, 2\xi_1\xi_2, 1-\xi_3^2\right) = (x_1,x_2,t) \in \VV_0^3. 
$$
To see that the mapping preserves the distance, we use the following identity 
$$
\f14 \left(\sqrt{(t+x_1)(s+y_1)} +\mathrm{sign}(x_2 y_2) \sqrt{(t-x_1)(s-y_1)}\right)^2 = \frac12(ts+x_1y_1 +x_2 y_2), 
$$
which was used to prove that $\sd_{\VV_0}$ is indeed a distance on $\VV_0^3$ in \cite[Proposition 4.2]{X21b},
where, however, the factor $\f14$ was missing.  
\end{proof}

A point on $\VV_0^3$ has three components, which we write as $(x,y,t)$ for $(x, y) \in \BB^2$ and $t \in [0,1]$.
For a set of points $\Xi_N = \{(x_i, y_i, t_i), 1 \le i \le N\}$ in $\VV_0^3$, the matrix $f[\Xi_N]$ is given by
$$
   f[\Xi_N] = \left[ \sqrt{\frac{t_i t_j + x_i x_j + y_i y_j}{2}} + \sqrt{1- t_i}  \sqrt{1- t_j} \right]_{i,j=1}^N. 
$$
Note that the elements of the matrix are all nonnegative. Since $\cos \d_{\VV_0}((x,t),(y,s)) \ge 0$, 
we only need the function $f$ on $[0,1]$. Instead, we shall assume that $f$ is an even function on $[-1,1]$. 
Then $f$ can be expanded as a series of $C_n^{\f12}(t) = P_n(t)$, where the $P_n$ are the Legendre 
polynomials. Using the distance preserving correspondence of $\VV_0^3$ and $\SS_{2,+}^2$, we can
derive the characterization of PDFs on $\VV_0^3$ from Theorem \ref{thm: PDF-Sk} with $\SS^2_{2,+}$. 

\begin{thm}\label{thm:PDF-V0}
 All functions $f: [-1,1] \mapsto \RR$ that are even on $[-1,1]$ and given by the expansion  
$$
 f(\cos \t) = \sum_{n=0}^\infty a_{2n} P_{2n} (\cos\t), \qquad a_{2n} \ge 0, \quad \forall n \in \NN_0,
$$
are in $\Phi(\VV_0^3)$. Moreover, such an $f$ is in $\mathrm{S}\Phi(\VV^3)$ if and only if there are infinitely 
many positive $a_{2n}$.
\end{thm}

On the $\VV_0^{d+1}$, the Chebyshev weight function is $w_{-1}$, defined by
$$
w_{-1} (t) = t^{-1}(1-t)^{-\f12}, \qquad  0 \le t \le 1.
$$
The orthogonal polynomials with respect to $w_{-1}(t)\, \d\s_{\VV_0}(x,t)$ are studied in \cite{X20}. Let 
$\CV_n(\VV_0^3)$ be the space of orthogonal polynomials of degree $n$ in $L^2(\VV_0^3, w_{-1})$. 
The reproducing kernel $\sP_n(w_{-1}; \cdot,\cdot)$ of the space $\CV_n(\VV_0^3)$ satisfies an 
addition formula given by \cite[(8.6)]{X20}, 
\begin{align}\label{eq:additionV0}
\sP_n\big(w_{-1}; (x,t), (y,s)\big) = \frac14 & \left[\sP_n^{+,+}\big((x,t), (y,s)\big) + \sP_n^{+,-}\big((x,t), (y,s)\big)\right.\\
         &  \left.{} + \sP_n^{-,+}\big((x,t), (y,s)\big) + \sP_n^{-,-}\big((x,t), (y,s)\big)\right], \notag
\end{align} 
where $\sP_n^{\pm, \pm}\big(\cdot,\cdot\big)$ is defined by, with $x = (x_1,x_2)$ and $y =(y_1,y_2)$,
\begin{align*}
\sP_n^{\pm,\pm} \big((x,t), (y,s)\big) = Z_{2n}^{\f{1}{2}} \left(\sqrt{ \frac{ts \pm (x_1 y_1 + x_2 y_2)}{2}}
      \pm \sqrt{1-t}\sqrt{1-s} \right). 
\end{align*}   
Recall that $Z_{n}^{\f12} (t)= (2n+1) C_n^{\f12} = (2n+1)P_n(t)$. 
Thus, once again, we see that $\sP_n^{+,+}(\cdot,\cdot)$ is a positive definite kernel on $\VV_0^3$ and 
the sum in \eqref{eq:additionV0} is an average that makes it a kernel of polynomials. Moreover, for each 
fixed $(x,t)$, the average is taken over $Z_{n}^{\f12}\big(\cos \sd_{\SS}(\xi(x,t)), \xi^{\pm} (y,s)\big)$, 
where $\xi^{\pm} = (\xi_1, \pm \xi_2, \pm \xi_3)$. 

Finally, let us mention that we do not know how to characterize the PDFs on the solid cone $\VV^3$ bounded
by $\VV_0^3$ and the hyperplane $t=1$. Making use of the characterization of $\Phi(\VV_0^3)$, we can 
obtain a characterization of $\Phi(\VV_0^2)$. However, $\VV_0^2$ is a triangle, which becomes the triangle
$\TT^2$ in the next subsection by an affine transformation. 

\subsection{PDF on the simplex}
The simplex $\TT^d$ of $\RR^d$ is defined by 
$$
  \TT^d = \left \{x \in \RR^d: x_1 \ge 0,\ldots, x_d \ge 0,\, |x| \le 1 \right\},
$$
where $|x| =x_1+\cdots + x_d$. The distance function on this domain is defined by  
$$
 \sd_\TT(x,y) = \arccos \left(\left\langle \sqrt{x}, \sqrt{y} \right \rangle + \sqrt{1- |x|}  \sqrt{1- |y|} \right),
$$
where $\sqrt{x} := (\sqrt{x_1}, \ldots, \sqrt{x_d})$. For a set of
points $\Xi_N = \{x_i: 1 \le i \le N\}$ in $\TT^d$,  
the matrix $f[\Xi_N]$ is then given by 
$$
   f[\Xi_N] = \left[ \left\langle \sqrt{x_i}, \sqrt{x_j} \right \rangle + \sqrt{1- |x_i|}  \sqrt{1- |x_j|}  \right]_{i,j=1}^N, 
$$
which is a matrix that has all of its elements nonnegative. It is known that there is a one-to-one correspondence
between $\TT^d$ and the positive quadrant of the unit sphere $\SS_{1,+}^d$, which is defined by 
$$
 \TT^d \ni x \mapsto \left(\sqrt{x_1}, \ldots, \sqrt{x_d}, \sqrt{1-|x|} \right)  = T(x) \in \SS_{1,+}^d
$$
and it preserves the distance between the points of two domains: 
$$
  \sd_\TT(x,y) = \arccos \la T(x), T(y)\ra =  \sd_{\SS} \left(T(x),T(y)\right).
$$
Hence, the characterization of PDF on the simplex follows from Theorem
\ref{thm: PDF-Sk}. As is the
case on the conic surface, $\cos \sd_\TT(x,y)$ is nonnegative so we only need $f$ to be supported
on $[0,1]$, which we extended evenly to $[-1,1]$. As in all previous cases in the section, we state a 
simplified version for the record. 

\begin{thm}\label{thm:PDF-T}
All functions $f: [-1,1] \mapsto \RR$ that are even on $[-1,1]$ and given by the expansion  
$$
 f(\cos \t) = \sum_{n=0}^\infty a_{2n}  C_{2n}^{\f{d-1}{2}}(\cos\t), \qquad a_{2n} \ge 0, \quad \forall n \in \NN_0,
$$
are in $\Phi(\TT^d)$. Moreover, such an $f$ is in $\mathrm{S}\Phi(\TT^d)$ if and only if there are 
infinitely many positive $a_{2n}$.
\end{thm}

The weight function on the simplex $\TT^d$ that corresponds to the surface measure is the Chebyshev weight 
function 
$$
  W_0(x) = \frac{1}{\sqrt{x_1 \cdots x_d (1-|x|)}}.
$$
Let $\CV_n(\TT^d)$ be the space of orthogonal polynomials in $L^2(\TT^d, W_0)$. Then the reproducing 
kernel $\sP_n(W_0;\cdot, \cdot)$ satisfies an addition formula given by \cite[(4.4.5)]{DX}
$$
 \sP_n(W_0; x,y) = \frac1{2^d} \sum_{\ve \in \{1,-1\}^d} Z_{2n}^{\frac{d-1}2}\left( \left \langle \sqrt{x}, \ve \sqrt{y} \right \rangle
        + \sqrt{1-|x|}\sqrt{1-|y|^2}\right).   
$$
As in all previous cases, the kernel is derived from the positive definite kernel $$Z_{2n}^{\frac{d-1}2}(\cos \sd_\TT(x,y))$$
by taking an average that ensures it is a kernel of polynomials. Moreover, the average is taken over
$Z_{2n}^{\frac{d-1}2}(\cos \sd_\TT(x, \ve y))$, where $\ve y$ is the reflection of $y$ over a coordinate plane.

\end{document}